\documentclass[draftclsnofoot,peerreview,onecolumn]{IEEEtran}
\usepackage[utf8]{inputenc} 
\usepackage[T1]{fontenc}    
\usepackage{url}            
\usepackage{booktabs}       
\usepackage{amsfonts}       
\usepackage{nicefrac}       
\usepackage{microtype}      

\usepackage{amsthm}

\usepackage{amsmath}
\usepackage{amssymb}
\usepackage[linesnumbered,ruled]{algorithm2e}
\usepackage{cite}

\newtheorem{lemma}{Lemma}
\newtheorem{theorem}{Theorem}
\newtheorem{corollary}{Corollary}
\newtheorem{definition}{Definition}

\newcommand{\RR}{\mathbb{R}}

\newcommand{\argmin}[1]{\operatorname{arg}\min_{#1}}
\newcommand{\argmax}[1]{\operatorname{arg}\max_{#1}}
\newcommand{\define}{\overset{\triangle}{=}}
\newcommand{\E}[2]{\mathbb{E}_{#1}\left[#2\right]}
\newcommand{\Cov}[2]{\Sigma_{#1}\left[#2\right]}

\newcommand{\func}{f}
\newcommand{\fsub}{g}
\newcommand{\argu}{x}
\newcommand{\altarg}{y}
\newcommand{\argset}{\mathcal{K}}
\newcommand{\fcount}{n}
\newcommand{\dimcnt}{d}
\newcommand{\heslow}{\alpha}
\newcommand{\hesupp}{\beta}
\newcommand{\cond}{\kappa}
\newcommand{\softer}{s}
\newcommand{\gradupp}{L}
\newcommand{\dist}{D}
\newcommand{\prob}{p}
\newcommand{\grad}{\nabla}
\newcommand{\hess}{\nabla^2}
\newcommand{\iden}{I}
\newcommand{\optgap}{\delta}
\newcommand{\arbsmall}{\varepsilon}
\newcommand{\cost}{c}
\newcommand{\point}{b}
\newcommand{\radius}{R}

\title{Accelerating Min-Max Optimization with Application to Minimal Bounding Sphere}

\author{
	Hakan~Gokcesu, 
	Kaan~Gokcesu, 
	and~Suleyman~S.~Kozat
	\thanks{
		H. Gokcesu and Suleyman S. Kozat are with the Department of Electrical and Electronics Engineering, Bilkent University, Ankara, Turkey; e-mail: \{hgokcesu, kozat\}@ee.bilkent.edu.tr, tel: +90 (312) 290-2336.
		\newline\indent
		K. Gokcesu is with the Department of Electrical Engineering and Computer Science, Massachusetts Institute of Technology, Cambridge, MA 02139 USA, e-mail: gokcesu@mit.edu.
	}
}

\begin{document}
\maketitle

\vspace*{\fill}
\begin{abstract}
	We study the min-max optimization problem where each function contributing to the max operation is strongly-convex and smooth with bounded gradient in the search domain. 
	By smoothing the max operator, we show the ability to achieve an arbitrarily small positive optimality gap of $\optgap$ in $\tilde{O}(1/\sqrt{\optgap})$ computational complexity (up to logarithmic factors) as opposed to the state-of-the-art strong-convexity computational requirement of $O(1/\optgap)$. We apply this important result to the well-known minimal bounding sphere problem and demonstrate that we can achieve a $(1+\arbsmall)$-approximation of the minimal bounding sphere, i.e. identify an hypersphere enclosing a total of $\fcount$ given points in the $\dimcnt$ dimensional unbounded space $\RR^\dimcnt$ with a radius at most $(1+\arbsmall)$ times the actual minimal bounding sphere radius for an arbitrarily small positive $\arbsmall$, in $\tilde{O}(\fcount\dimcnt/\sqrt{\arbsmall})$ computational time 
	as opposed to the state-of-the-art approach of core-set methodology, which needs $O(\fcount\dimcnt/\arbsmall)$ computational time.
\end{abstract}
\vspace*{\fill}

\newpage
\section{Introduction}
The min-max optimization has been extensively studied in the literature due to its wide range of applications. It generally appears in the fields of statistics, operations research and engineering under the topics of throttling, resource allocation, computer graphics, computational geometry, clustering, anomaly detection and facility location.

There has been attempts to solve this problem via alternative formulations or smoothing approximations in \cite{Chen, Zang, Fang, Zhao}.
The technique of solving the min-max optimization by smoothing the target has been extensively studied in the literature. 
However, even after these extensive studies, the convergence rate analysis is very limited in literature and existing works generally 
try to show that they converge to the optimal solution given enough time (existence proofs). 
An example to the limited convergence rate analysis regarding this problem could be found in  \cite{Nesterov} which solves the generic non-smooth min-max optimization.

Consequently, to our knowledge, as the first time in the literature, we have derived a major improvement on the convergence guarantees for the min-max optimization problems where the components contributing to the max operation are strongly-convex, smooth and have bounded gradients. Our convergence rate is such that for an optimality gap of $\optgap$, we need $O(\sqrt{1/\optgap})$ computational resource improving upon the optimization complexity of $O(1/\optgap)$ for non-smooth strongly-convex functions having bounded gradients \cite{Lacoste}.

A specific type of this widely studied optimization problem is named the minimal bounding sphere. There have been several attempts to solve this problem deterministically. The computational complexity is generally super-linear with respect to the number of points and the vector space dimensions such that the dependency is polynomial with integer powers, at times, much larger than 1. Thus, the feasible approaches to this problem generally focus on heuristic methods with experimentally shown efficiency \cite{Larsson,Welzl,Fischer}.

An alternative approach to finding minimal bounding spheres, with linear time-complexity dependencies with respect to the number of point and the vector space dimension is the so-called $(1+\arbsmall)$-approximation. The corresponding attempts are based on the core-set constructions \cite{Badoiu,Kumar}. The state-of-the-art of these approximative solutions has been able to find a bounding sphere with radius $(1+\arbsmall)\radius$, where $\radius$ denotes the actual minimal bounding sphere radius, with time-complexity 
$O(\fcount\dimcnt/\arbsmall)$ \cite{Kumar} for arbitrarily large number of points $\fcount$ and vector space dimension $\dimcnt$. We improve upon this by showing that the time-complexity can be reduced to $\tilde{O}(\fcount\dimcnt/\sqrt{\arbsmall})$ (up to logarithmic factors).

We next continue with a rigorous formulation of the problem, after which we demonstrate how the improvements for both the general min-max optimization and the minimal bounding sphere are achieved.

\subsection{Problem Description for the General Min-Max Optimization}
Our convex optimization problem is such that the function $\func$ to be minimized is of the form:
\begin{equation} \label{eq:func}
	\func(\argu) \define \max_{1 \leq i \leq \fcount} \func_i(\argu),
\end{equation}
where $\fcount$ is the number of functions amongst which we select the maximum for a given argument $\argu \in \RR^\dimcnt$ via the $\max$ operator. Each function $\func_i$ is twice-differentiable and displays strong-convexity with Lipschitz-smoothness. The gradients are also assumed to be bounded at least in the subspace of $\RR^\dimcnt$ subjected to the iterative search including the optimal point. This subspace can be either preset or naturally occurring due to the nature of our method.
Twice differentiability is required since our analysis depends on behavior of the Hessian matrix.

Normally, $\func(\cdot)$ in \eqref{eq:func} is non-smooth due to the maximum operator. However, we will show that by optimizing a substitute function, which approximates the original $\func(\cdot)$ sufficiently well, we can improve the time dependency of the regular convergence rate from $O(1/t)$ to $O(1/t^2)$, where $O(\cdot)$ is the big-O notation.

After we finish our discussion on the general setting, we will investigate a special case of this min-max optimization called minimal bounding sphere 
in Section \ref{sec:mbs}.

\section{Smooth Approximation of the Max Operator}
Let us define the new function $\fsub_\softer(\argu)$, which we shall use as substitute for $\func(\argu)$, as follows:
\begin{equation} \label{eq:fsub}
	\fsub_\softer(\argu) = \frac{1}{\softer} \log 
	\left(
		\sum_{i=1}^{\fcount} \exp(\softer \func_i(\argu))
	\right),
\end{equation}
where $\log(\cdot)$ is the natural-logarithm, $\exp(\cdot)$ is the natural-exponentiation and $\softer > 0$. This form of smooth maximum is also referred to as "LogExpSum". We now present a lemma regarding how well $\fsub_\softer(\cdot)$ approximates $\func(\cdot)$.
\begin{lemma} \label{lemma:fsub}
	The substitute function $\fsub_\softer(\cdot)$ is both lower and upper bounded by $\func(\cdot)$ with the upper bound having an additive redundancy of at most $\softer^{-1}\log\fcount$ such that
	\begin{equation*}
		\func(\argu) \leq \fsub_\softer(\argu) \leq \func(\argu) + \frac{\log \fcount}{\softer}.
	\end{equation*} 
\end{lemma}
\begin{proof}
	By the definitions of $\log$ and $\exp$, we have,
	\begin{equation} \label{eq:f_alt}
		\func(\argu) = \frac{1}{\softer} \log(\exp(\softer f(\argu))) = \frac{1}{\softer} \log(\exp(\softer f_j(\argu))),
	\end{equation}
	where $j = \argmax{1\leq j\leq \fcount} \func_j(\argu)$  due to \eqref{eq:func}.
	Since $\log$ is a monotonically increasing function and $\exp(p) \geq 0$ for all $p \in \RR$, the combination of \eqref{eq:fsub} and \eqref{eq:f_alt} yields,
	\begin{equation} \label{eq:f_low}
		\func(\argu) \leq \frac{1}{\softer} \log 
		\left(
		\sum_{i=1}^{\fcount} \exp(\softer \func_i(\argu))
		\right) = \fsub_\softer(\argu).
	\end{equation}
	Again, due to monotonicity of $\log$ and $\exp$, we can replace each individual $\func_i(\cdot)$ in \eqref{eq:f_low} with $\func(\cdot)$ as an upper-bound. In combination with \eqref{eq:f_low}, this would result in the lemma.
\end{proof}

Lemma \ref{lemma:fsub} implies that if we optimize $\fsub_\softer(\cdot)$ instead, we would incur an additional redundancy of at most $(\log \fcount)/\softer$ as a cost for smoothing the target function.

\begin{corollary} \label{cor:gap}
	The gap between $\func(\argu)$ and $\func(\altarg)$ can be decomposed into a "smoothing" regret and the gap between their smoothed counterparts as follows:
	\begin{equation*}
		\func(\argu) - \func(\altarg) \leq
		\frac{\log\fcount}{\softer} +  \left[\fsub_\softer(\argu) - \fsub_\softer(\altarg)\right],
	\end{equation*}
	where the "smoothing" regret is $(\log\fcount)/\softer$.
\end{corollary}
\begin{proof}
	The result follows directly from Lemma \ref{lemma:fsub}
\end{proof}

We introduce the short-hand notation for the optimal point minimizing $\func(\cdot)$ as:
\begin{equation} \label{eq:star}
\argu^* \define \argmin{\argu \in \RR^\dimcnt} \func(\argu).
\end{equation}

Next, we investigate derive some properties of this new function $\fsub_\softer(\cdot)$, namely the gradient and Hessian, after which we can investigate its strong-convexity and smoothness parameters.

\subsection{The gradient and the Hessian of the substitute function}
We start with a probability vector definition, which is used for writing weighted sums via expectations.
\begin{definition} \label{def:prob}
	Given the "smoother" $\softer$ and the argument $\argu$, we generate the probability vector $\prob_\softer(\argu)$ such that:
	\begin{equation*}
		\prob_{\softer,i}(\argu) = \frac{\exp(\softer \func_i(\argu))}{\sum_{j=1}^\fcount \exp(\softer \func_j(\argu))},
	\end{equation*}
	where $\prob_{\softer,i}(\argu)$ is the $i^{th}$ element of the vector $\prob_\softer(\argu)$.
\end{definition}

In the following lemmas, we compute the gradient and, from there, the Hessian of substitute function $\fsub_\softer(\cdot)$, which are used for the iterative optimization. 
\begin{lemma} \label{lemma:sub_grad}
	We can write the gradient $\grad \fsub_\softer(\argu)$ as a weighted combination of individual gradients $\grad \func_i(\argu)$ where the weights sum to $1$ such that
	\begin{equation*}
		\grad \fsub_\softer(\argu) = \E{\prob_\softer(\argu)}{\grad \func_i(\argu)},
	\end{equation*}
	where $\E{\prob_\softer(\argu)}{\cdot}$ is the expectation operation with respect to the probability mass function corresponding to the size-$\fcount$ vector $\prob_\softer(\argu)$. Each element $\prob_{\softer,i}(\argu)$ of $\prob_\softer(\argu)$ corresponds to the probability 
	assigned to $\grad \func_i(\argu)$ 
	as defined in Definition \ref{def:prob}.
\end{lemma}
\begin{proof}
	The result directly follows after taking the partial derivatives of \eqref{eq:fsub} with respect to each element in $\argu$.
\end{proof}

\begin{lemma} \label{lemma:sub_hess}
	Considering the gradient $\grad \func_i(\argu)$ as a random vector and the Hessian $\hess \func_i(\argu)$ as a random matrix, each having $\fcount$ possible realizations generated from the probability mass function corresponding to the vector $\prob_\softer(\argu)$, the Hessian $\hess \fsub_\softer(\argu)$ can be computed with the expectation of $\hess \func_i(\argu)$ and the covariance matrix of $\grad \func_i(\argu)$ as follows:
	\begin{equation*}
		\hess \fsub_\softer(\argu) = 
		\softer \left(
			\Cov{\prob_\softer(\argu)}{\grad \func_i(\argu)}
		\right)
		+ \E{\prob_\softer(\argu)}{\hess \func_i(\argu)},
	\end{equation*}
	where $\E{\prob_\softer(\argu)}{\cdot}$, $\prob_\softer(\argu)$ are defined as in Lemma \ref{lemma:sub_grad} and the covariance matrix 
	is given as,
	\begin{equation} \label{eq:cov}
		\Cov{\prob_\softer(\argu)}{\grad \func_i(\argu)}
		\define \E{\prob_\softer(\argu)}{\grad \func_i(\argu) \grad \func_i(\argu)^T} 
		- \E{\prob_\softer(\argu)}{\grad \func_i(\argu)}\E{\prob_\softer(\argu)}{\grad \func_i(\argu)}^T.
	\end{equation}
\end{lemma}
\begin{proof}
	The result directly follows from taking further partial derivatives of gradient in Lemma \ref{lemma:sub_grad}.
\end{proof}

In the following section, we explain our methodology for accelerating the convergence rate.

\section{Accelerated Optimization of the 
	Approximation} \label{sec:acce}
We utilize Nesterov's accelerated gradient descent method for smooth and strongly-convex functions, for which more details are given in \cite{Bubeck}. The algorithm is an iterative one, where the iterations are done in an alternating fashion. Starting with the initial argument pair $\argu_1 = \altarg_1$, we have the following iterative relations for $\argu_t$ and $\altarg_t$ for $t \geq 1$:
\begin{equation} \label{eq:acce_grad}
	\begin{aligned}
		\argu_{t+1} &= \altarg_t - \frac{1}{\hesupp_\softer} \grad \fsub_\softer(\altarg_t),\\
		\altarg_{t+1} &= \argu_{t+1} + \frac{\sqrt{\cond_\softer}-1}{\sqrt{\cond_\softer}+1} \left(\argu_{t+1} - \argu_t\right)
		,
	\end{aligned}
\end{equation}
with $\cond_\softer$ being the condition number of the Hessian $\hess \fsub_\softer(x)$ in Lemma \ref{lemma:hess_bound}, which is computed as
\begin{equation} \label{eq:hess_bound}
	\begin{gathered}
		\cond_\softer = \hesupp_\softer/\heslow_\softer,
		\\\text{ for }
		\heslow_\softer \iden 
		\preceq \hess \fsub_\softer(\argu) 
		\preceq \hesupp_\softer \iden, 
		\text{ for all }
		\argu \in \argset_\softer,
	\end{gathered}
\end{equation}
where $\heslow_\softer$ and $\hesupp_\softer$ are the lower and upper bounds on the eigenvalues of Hessian $\hess \fsub_\softer(\argu)$, respectively, the identity matrix of $\dimcnt \times \dimcnt$ dimensions is denoted as $\iden$, and $\argset_\softer \subseteq \RR^\dimcnt$ is a set guaranteed to include the convex-hull of all iterations $\{\argu_t\}_{t=1}^\infty$, $\{\altarg_t\}_{t=1}^\infty$ and the optimal point $\argu^*$ as defined in \eqref{eq:star}. 

Generating the Hessian upper-bound (the smoothness parameter) $\hesupp_\softer$, and consequently the condition number $\cond_\softer$ for the set $\argset_\softer$ is sufficient as in \eqref{eq:hess_bound}. The reason is twofold. Firstly, the optimality gap guarantee shown in the following as Lemma \ref{lemma:opt_gap} is dependent upon upper-bounding the Hessian on line segments pair-wise connecting the algorithm iterations ($\{\argu_t\}_{t=1}^\infty$, $\{\altarg_t\}_{t=1}^\infty$) and the optimal point $\argu^*$ via $\hesupp_\softer$. All of such segments are encapsulated by the convex-hull of $\{\argu_t\}_{t=1}^\infty$, $\{\altarg_t\}_{t=1}^\infty$ and the optimal point $\argu^*$. Secondly, this convex-hull is itself a  subset of $\argset_\softer$ as previously defined.
\begin{lemma} \label{lemma:opt_gap}
	The following optimality gap is guaranteed for $\argu_t$:
	\begin{equation*}
		\fsub_\softer(\argu_t) - \fsub_\softer(\argu^*) 
		\leq 
		\left(
			\frac{\heslow_\softer}{2} \|\argu_1-\argu^*\|^2 + \fsub_\softer(\argu_1) - \fsub_\softer(\argu^*)
		\right)
		\exp\left(
			-\frac{t-1}{\sqrt{\cond_\softer}}
		\right),
	\end{equation*}
	where $\cond_\softer = \hesupp_\softer / \heslow_\softer$ is the condition number. $\hesupp_\softer$ and $\heslow_\softer$ are the strong-convexity and Lipschitz-smoothness parameters, respectively, and $\argu^*$ is the optimal point as defined in \eqref{eq:star}.
\end{lemma}
\begin{proof}
	The proof directly follows a similar formulation given in \cite{Bubeck} under "the smooth and strong convex case" subsection of the section "Nesterov's accelerated gradient descent". The only exception is that we do not replace $\fsub_\softer(\argu_1) - \fsub_\softer(\argu^*)$ with an upper-bound and leave it as is.
\end{proof}

\subsection{Parameters of Strong-Convexity 
	and Lipschitz-Smoothness 
}
To compute $\cond_\softer$, we bound the eigenvalues of $\hess \fsub_\softer(\argu)$.
\begin{lemma} \label{lemma:hess_bound}
	We can lower and upper bound eigenvalues of the Hessian matrix $\hess \fsub_\softer(\argu)$ for $\argu \in \argset_\softer$ as follows:
	\begin{equation*}
		\left(
		\min_{1\leq i\leq \fcount} \heslow_{\softer,i}
		\right) \iden
		\preceq \hess \fsub_\softer(\argu)
		\preceq \left(
		\softer \gradupp_\softer^2 + \max_{1\leq i\leq \fcount} \hesupp_{\softer,i}
		\right) \iden
		,
	\end{equation*}
	where $\heslow_{\softer,i}$ and $\hesupp_{\softer,i}$ are further defined as the strong-convexity and smoothness parameters for the components from the "$\max_{1\leq i\leq \fcount} \func_i(\cdot)$" operator generating $\func(\argu)$, i.e. $\func_i(\argu)$, respectively, such that we have $\heslow_{\softer,i} \iden \preceq \hess \func_i(\argu) \preceq \hesupp_{\softer,i} \iden$, for $\argu \in \argset_\softer$. The parameter $\gradupp_\softer$ is a common gradient norm bound for each $\func_i(\cdot)$ such that $\gradupp_\softer \geq \grad \func_i(\argu)$ for each $1\leq i\leq \fcount$ and $\argu \in \argset_\softer$.
\end{lemma}
\begin{proof}
	We start with proving the lower-bound relation. Using 
	Lemma \ref{lemma:sub_hess}, we obtain
	\begin{equation*}
		\hess \fsub_\softer(\argu) 
		\succeq \E{\prob_\softer(\argu)}{\hess \func_i(\argu)}
		,
	\end{equation*}
	since the covariance matrix $\Cov{\prob_\softer(\argu)}{\grad \func_i(\argu)}$ is lower-bounded by $0$ as it is a convex combination of rank-$1$ self-outer-product matrices with their lowest eigenvalue being $0$.
	
	The expectation operation $\E{\prob_\softer(\argu)}{\cdot}$ is linear. Thus we can replace each $\hess \func_i(\argu)$ with its lower-bound $\heslow_{\softer,i} \iden$ without affecting the inequality relation $\succeq$. After taking the constant identity matrix $\iden$ outside of the expectation, we have the renewed relation
	\begin{equation*}
		\hess \fsub_\softer(\argu) 
		\succeq \E{\prob_\softer(\argu)}{\heslow_{\softer,i}} \times \iden
		.
	\end{equation*}
	Since the expectation is a convex combination of scalars $\heslow_{\softer,i}$, we further lower bound by replacing the expectation with $\left(\min_{1\leq i\leq \fcount} \heslow_{\softer,i}\right)$, which gives the lower-bound of this lemma.
	
	For the upper-bound, we can generate
	\begin{equation*}
		\hess \fsub_\softer(\argu) 
		\preceq \left(
		\max_{1\leq i\leq \fcount} \hesupp_{\softer,i}
		\right) \iden
		+ \softer \left(
		\Cov{\prob_\softer(\argu)}{\grad \func_i(\argu)}
		\right)
	\end{equation*}
	using Lemma \ref{lemma:sub_hess} by upper bounding each $\hess \func_i(\argu)$ with $\hesupp_{\softer,i} \iden$ and the resulting expectation with $\max_{1\leq i\leq \fcount} \hesupp_{\softer,i}$ similar to the lower-bound.
	
	We can upper bound the covariance matrix by first noting that the eigenvalues of an $\dimcnt$-dimensional outer-product $uv^T$ are $u^Tv$ and $\dimcnt-1$ zeros. Consequently, we upper bound it by replacing the negative outer-product, i.e. $-\E{\prob_{\softer}(\argu)}{\grad \func_i(\argu)} \E{\prob_{\softer}(\argu)}{\grad \func_i(\argu)}^T$, in \eqref{eq:cov} with $0$. Then, utilizing the linearity of expectation again, we get the final upper-bound by replacing the outer-product $\grad \func_i(\argu) \grad \func_i(\argu)^T$ inside expectation with $\|\grad \func_i(\argu)\|^2 \iden$. The resulting upper-bound is given as
	\begin{equation*}
		\hess \fsub_\softer(\argu) 
		\preceq \left(
		\max_{1\leq i\leq \fcount} \hesupp_{\softer,i}
		\right) \iden
		+ \softer \left(
		\E{\prob_{\softer}(\argu)}{\|\grad \func_i(\argu)\|^2}
		\right) \iden
		,
	\end{equation*}
	after taking the constant identity matrix $\iden$ outside of expectation. We can replace the scalar $\E{\prob_{\softer}(\argu)}{\|\grad \func_i(\argu)\|^2}$ with a common squared gradient-norm bound $\gradupp_\softer^2$, which gives the upper-bound relation of this lemma, thus concluding the proof.
\end{proof}

\subsection{Algorithm Description}

We start at some point $\argu_1$. 
We determine the "smoother" $\softer$ needed to achieve the requested optimality gap $\optgap$ and the set $\argset_\softer \subseteq \RR^\dimcnt$ such that $\argset_\softer$ includes the optimal point $\argu^*$ and all future iterations $\{\argu_t\}_{t=1}^\infty$, $\{\altarg_t\}_{t=1}^\infty$. We use the update rules in \eqref{eq:acce_grad} after determining the common gradient norm bound $\gradupp_\softer$, the individual strong-convexity and Lipschitz-smoothness parameters $\{\heslow_{\softer,i}\}_{i=1}^\fcount$ and $\{\hesupp_{\softer,i}\}_{i=1}^\fcount$, respectively, via the set $\argset_\softer$. The condition number $\cond_\softer$ and the smoothness parameter $\hesupp_\softer$ are calculated using the lower and upper bounds in Lemma \ref{lemma:hess_bound}. The pseudo-code is given in Algorithm \ref{alg:smooth}. For this algorithm, we have the following performance result.
\begin{algorithm}[t] \label{alg:smooth}
	\caption{Min-Max Optimizer} 
	\SetKwInOut{Input}{Input}
	\SetKwInOut{Output}{Output}
	\SetKwInOut{Initialization}{Initialization}
	\SetKwInOut{RunTime}{Run-Time}
	\Input{$\func(\cdot)$ - the optimization target  such that $\func(\argu) = \max_{1\leq i\leq \fcount} \func_i(\argu)$ for $\argu \in \RR^\dimcnt$,\newline
	$\optgap$ - requested optimality gap guarantee. 
	}
	\Output{$\argu_t$ for $t\geq1$.}
	\vspace{8pt}
	\Initialization{}
	Set $s = 2\optgap^{-1}\log\fcount$ as the "smoother" for $\func(\cdot)$.
	\\Generate $\fsub_\softer(\argu) = \softer^{-1} \log 
	\left(
	\sum_{i=1}^{\fcount} \exp(\softer \func_i(\argu))
	\right)$, the smooth approximation of $\func(\cdot)$.
	\\Determine $\argset_\softer \subseteq \RR^\dimcnt$, the subset containing optimal point and all future iterations.
	\\Determine $\{\heslow_{\softer,i}\}_{i=1}^\fcount$, the strong-convexity parameters for each $\func_i(\cdot)$ paired with domain $\argset_\softer$.
	\\Set $\heslow_{\softer} = \min_{1\leq i\leq \fcount} \heslow_{\softer,i}$ as the strong-convexity parameter of $\fsub_\softer(\cdot)$.
	\\Determine $\{\hesupp_{\softer,i}\}_{i=1}^\fcount$, the smoothness parameters for each $\func_i(\cdot)$ paired with domain $\argset_\softer$.
	\\Determine $\gradupp_\softer$, the common gradient norm bound for all $\func_i(\cdot)$ paired with domain $\argset_\softer$.
	\\Set $\hesupp_{\softer} = \softer \gradupp_\softer^2 + \max_{1\leq i\leq \fcount} \hesupp_{\softer,i}$ as the smoothness parameter of $\fsub_\softer$.
	\\Calculate $\cond_\softer = \hesupp_\softer/\heslow_\softer$, the condition number of $\fsub_\softer(\cdot)$. 
	\\\vspace{8pt}
	\RunTime{}
	Initialize $\argu_1$, e.g. $\argu_1 = 0$.\\
	Set $\altarg_1 = \argu_1$.\\
	Start with $t=1$.\\
	\While{not terminated by user}{
		Calculate $\prob_\softer(\altarg_t)$, the probability vector, from Definition \ref{def:prob} via $\softer$ and argument $\altarg_t$.\\
		$\argu_{t+1} = \altarg_t - \hesupp_\softer^{-1}\; 
		\E{\prob_\softer(\altarg_t)}{\grad \func_i(\altarg_t)}
		$, using Lemma \ref{lemma:sub_grad} and \eqref{eq:acce_grad}\\ \label{line:x}
		$\altarg_{t+1} = \argu_{t+1} + \left(
			1 - 2\left(\sqrt{\cond_\softer} + 1\right)^{-1}
		\right) \left(\argu_{t+1} - \argu_t\right)$
		\\ \label{line:y}
		$t \leftarrow t+1$.
	}
\end{algorithm}

\begin{theorem} \label{thm:itecnt}
	We run Algorithm \ref{alg:smooth} for a given optimality gap guarantee $\optgap$.
	Then, we achieve the gap $\func(\argu_t) - \func(\argu^*) \leq \optgap$ after sufficient iterations $t$ such that:
	\begin{equation*}
		\begin{gathered}
			t \in O\left(\sqrt{\frac{\log\fcount}{\optgap}} \log \frac{1}{\optgap}\right),
			\text{ since }
			\\t = 1 + \sqrt{\frac{2}{\optgap}\frac{\gradupp_s^2 \log \fcount}{\heslow_\softer} + \frac{ \tilde{\hesupp}_{\softer}}{\heslow_{\softer}}} 
			\log \left(
				\frac{1}{\optgap} \left(
					\heslow_\softer \dist_\softer^2 + 2\gradupp_\softer\dist_\softer
				\right)
			\right),
		\end{gathered}
	\end{equation*}
	where $O(\cdot)$ is the big-O notation for asymptotic 
	upper-bounding, $\fcount$ is the number of functions $\func_i(\cdot)$ contributing to the $\max$ operation resulting in $\func(\cdot)$, $\gradupp_\softer$ is the common gradient norm bound for each component function $\func_i(\argu)$ in the $\max$ operator 
	such that $\|\grad \func_i(\argu)\| \leq \gradupp_\softer$, for all $1\leq i\leq \fcount$, $\argu \in \argset_\softer$.
	$\heslow_{\softer}$ is the strong-convexity parameter of the approximation $\fsub_\softer(\argu)$, $\tilde{\hesupp}_\softer = \max_{1\leq i\leq \fcount} \hesupp_{\softer,i}$ is the pseudo-smoothness parameter upper bounding the matrix $\E{\prob_{\softer}(\argu)}{\hess \fsub_\softer(\argu)}$, and $\dist_\softer$ is the unknown initial distance between $\argu_1$ and $\argu^*$.
\end{theorem}
\begin{proof}
	From Lemma \ref{lemma:opt_gap}, we see that the lower $\cond_\softer$ results in faster convergence for a fixed optimality gap. Without further information on the gradient and Hessian bounds, we need to lower the "smoother" $\softer$ for a lower $\cond_\softer$. However, the "smoothing" regret $\softer^{-1}\log\fcount$ from Corollary \ref{cor:gap} works in the opposite direction. Consequently, we will equate both the optimality gap from the smooth approximation $\fsub_\softer(\cdot)$ and the "smoothing" regret to $\optgap/2$.
	This results in $\softer = 2\optgap^{-1}\log \fcount$, with $\fcount$ being the number of function $\func_i(\cdot)$ contributing to the same $\max$ operation. $\argset_\softer$ is generated consequently. Immediately, we have the "smoothing" regret in Corollary \ref{cor:gap} as $\optgap/2$. Then, we equate the gap from $\fsub_\softer(\cdot)$ using the upper-bound in Lemma \ref{lemma:opt_gap}. Afterwards, we replace the condition number $\cond_\softer$ in accordance with \eqref{eq:hess_bound} after calculating the strong-convexity and smoothness parameters $\heslow_\softer$ and $\hesupp_\softer$ via Lemma \ref{lemma:hess_bound}. Finally, we upper bound the initial smooth approximation gap $\fsub_\softer(\argu_1) - \fsub_\softer(\argu^*)$ with $\gradupp_\softer \dist_\softer$ using the convexity relation and arrive at the result of the theorem.
\end{proof}

\subsubsection{Computational Cost of the Algorithm}
\begin{corollary} \label{cor:time_comp}
	For an optimality gap $\optgap$, the computation time $T$ needed is such that $T\in O(\fcount \sqrt{\log\fcount/\optgap} \log^2(1/\optgap))$ for an arbitrarily small $\optgap > 0$. More specifically:
	\begin{equation*}
		T \in O\left(
			\fcount \sqrt{\frac{\log\fcount}{\optgap}} \log\frac{1}{\optgap} \left(\cost\dimcnt + \log\frac{1}{\optgap} + \log\log\fcount\right)
		\right),
	\end{equation*}
	where $\cost$ is the average cost of calculating a partial derivative for any $\func_i(\cdot)$, $\fcount$ is the number of functions contributing to $\func(\cdot)$ and $\dimcnt$ is the dimension of the domain of $\func(\cdot)$'s .
\end{corollary}
\begin{proof}
	We need $t \in O((\log\fcount/\optgap)^{1/2}\log(\optgap^{-1}))$ iterations as shown in Theorem \ref{thm:itecnt}. We observe that each iteration of the while-loop in Algorithm \ref{alg:smooth} requires $O(\fcount \dimcnt)$ partial derivative calculations. Due to the computation of probability vector $\prob_\softer(\argu_t)$ with respect to Definition \ref{def:prob}, each iteration also requires a total of $\fcount$ exponentiation to the power of $O(\optgap^{-1}\log\fcount)$ when $s = 2\optgap^{-1} \log n$. Each of such exponentiations has additional computational cost of $\left(\log(\optgap^{-1})+\log\log\fcount\right)$. Combination of these costs gives the corollary.
\end{proof}
\subsubsection{Online Version of the Algorithm (without Specifying $\optgap$)}
\begin{corollary}
	We can achieve the time-complexity in Corollary \ref{cor:time_comp}, which is of the form $T\in \tilde{O}(\cost\fcount\dimcnt \sqrt{\optgap^{-1}})$, in an online fashion with no requested optimality gap guarantee $\optgap$. $\tilde{O}$ is the soft-O notation ignoring logarithmic factors compared to big-O.
\end{corollary}
\begin{proof}
	We initialize with some $\optgap_0$ and run Algorithm \ref{alg:smooth} with $\optgap_0$ as the optimality guarantee. Then, after sufficient iterations to achieve the requested $\optgap_0$, we restart Algorithm \ref{alg:smooth} with a new guarantee $\optgap_{k} = \optgap_{k-1}/2$, for $k\geq1$ and repeat non-stop.
	
	For $\optgap$ such that $2^{1-m} \geq \optgap/\optgap_0 \geq 2^{-m}$ for some integer $m\geq1$, the total exhausted time can be upper-bounded as follows using the fact that $\log$ is monotonically increasing and $2^{-m}\optgap_0$ is lower-bounded with $\optgap/2$, 
	\begin{equation*}
		T \in O\left(
		\fcount 
		\left(\sum_{k=0}^{m} \sqrt{\frac{\log\fcount}{2^{-k} \optgap_0}}\right)
		\log\frac{2}{\optgap} \left(\cost\dimcnt + \log\frac{2}{\optgap}\log\log\fcount\right)
		\right)
	\end{equation*}
	This bound translates to the same bound in Corollary \ref{cor:time_comp}.
\end{proof}

In the next section, we shall investigate an interesting specific application for the general accelerated min-max optimization via smooth approximation, which we have introduced.
\section{(1+$\arbsmall$)-Approximation 
	for the Problem of Minimal Bounding Sphere}
\label{sec:mbs}
Let us suppose we have $\fcount$ points, each located at $\point_i$ for $1\leq i\leq \fcount$, in the $\dimcnt$ dimensional space $\RR^\dimcnt$. 
Our minimization target is such that:
\begin{equation} \label{eq:func_mbs}
	\func(\argu) = \max_{1\leq i\leq \fcount} 
	\|\argu - \point_i\|^2.
\end{equation}
This is the so-called minimal bounding sphere problem such that it finds an optimal point $\argu^*$, which, together with $\func(\argu^*)$ from \eqref{eq:func_mbs}, defines the center and radius of a ball enclosing all of the $\fcount$ points in $\RR^\dimcnt$ with the smallest possible radius. 

The optimal point $\argu^*$ is defined as:
\begin{equation*}
	\argu^* \define \argmin{\argu \in \RR^\dimcnt}{\func(\argu)}.
\end{equation*}

Since $\argu^*$ minimizes the maximum 
euclidean distance to a point $\point_i$, we know that $\argu^*$ belongs to the convex-hull of points $\{\point_i\}_{i=1}^\fcount$ since we can always decrease these 
distances by moving towards the convex-hull.

We shall utilize Algorithm \ref{alg:smooth} with the initial point $\argu_1$ belonging to this convex-hull, e.g. $\argu_1 = \fcount^{-1} \sum_{i=1}^\fcount \point_i$, the arithmetic mean of the points.

Before running Algorithm \ref{alg:smooth}, we determine the strong-convexity and Lipschitz-smoothness parameters, which are $\heslow_{\softer,i} = \hesupp_{\softer,i} = 2$
for all $1\leq i\leq \fcount$ in this particular problem. Consequently, the overall strong-convexity and pseudo-smoothness parameters are also 
$\heslow_\softer = \tilde{\hesupp}_\softer = 2$
, respectively. $\heslow_\softer$ reveals itself after combining Lemma \ref{lemma:hess_bound} with \eqref{eq:hess_bound}, and $\tilde{\hesupp}_\softer$ is defined in Theorem \ref{thm:itecnt} as the maximum smoothness parameter from the individual functions. What only remains to be set in Algorithm \ref{alg:smooth} is the gradient norm upper-bound $\gradupp_\softer$ which inherently includes determining the set $\argset_\softer$ guaranteed to include the optimal point $\argu^*$, and all iterations $\{\argu\}_{t=1}^\infty$, $\{\altarg\}_{t=1}^\infty$.

\subsection{Gradient Norm Bound 
	for Minimal Bounding Sphere}
Assume the minimal bounding sphere is such that the maximum distance (i.e. radius) between the optimal point $\argu^*$ and one of the other points $\point_i$ is $\radius$.

\begin{lemma} \label{lemma:radius_bounds}
	After setting the initial point $\argu_1$ and computing $\func(\argu_1)$ using \eqref{eq:func_mbs}, we have the following bounds on the minimal bounding sphere radius $\radius$:
	\begin{equation*}
	\sqrt{\func(\argu_1)}/2 \leq \radius \leq \sqrt{\func(\argu_1)}.
	\end{equation*}
\end{lemma}
\begin{proof}
	The upper-bound is trivial since $\argu_1$ is not necessarily optimal. The lower-bound comes from the fact that $\argu_1$ belongs to the convex-hull of $\{\point_i\}_{i=1}^\fcount$ and, consequently, $\func(\argu_1)$ cannot exceed 
	the diameter of minimal bounding sphere which encloses all points $\point_i$ and, hence, their convex-hull.
\end{proof}

\begin{lemma} \label{lemma:grad_upp}
	The gradient norm upper-bound is such that:
	\begin{equation*}
		\gradupp_\softer = 6\sqrt{5\func(\argu_1)+\optgap/2},
	\end{equation*}
	where $\argu_1$ is the initial point of Algorithm \ref{alg:smooth} and $\optgap$ is the requested optimality gap.
\end{lemma}
\begin{proof}
	In accordance with this specific problem, we can further upper bound the smooth approximation optimality guarantee in Lemma \ref{lemma:opt_gap} by first upper bounding the multiplicand in parenthesis on the greater side of the inequality since we have an exponential multiplier, i.e. $\exp(-(t-1)/\sqrt{\cond_\softer})$, which is guaranteed to be non-negative. After also upper bounding this exponential multiplier, since the upper bound of multiplicand turns out to be always nonnegative, we obtain the following result:
	\begin{equation} \label{eq:fsub_upp}
		\fsub_\softer(\argu_t) \leq 5 \radius^2 + \frac{\log \fcount}{s},
		\quad \text{ for all } t\geq 1.
	\end{equation}
	This upper bounding takes place by replacing the quantities in Lemma \ref{lemma:opt_gap} with their corresponding bounds using the facts $\heslow_\softer=2$, $\|\argu_1 - \argu^*\|\leq\radius$, $\fsub_\softer(\argu_1) \leq 4\radius^2 +\softer^{-1}\log \fcount$, and $\exp(-(t-1)/\sqrt{\cond_\softer}) \leq 1$ for all $t\geq 1$. The distance inequality $\|\argu_1 - \argu^*\|\leq\radius$ is due to a fact that the minimal bounding sphere has its center at $\argu^*$, and $\argu_1$ is contained inside the said sphere since it is encapsulated by the convex-hull of all the points $\point_i$. Similarly, the inequality $\fsub_\softer(\argu_1) \leq 4\radius^2 + \softer^{-1}\log \fcount$ results from Lemma \ref{lemma:fsub} and $\func(\argu_1) \leq 4\radius^2$, since $\argu_1$ is again contained in the same minimal bounding sphere with diameter $2\radius$.

	Then, by Lemma \ref{lemma:fsub}, \eqref{eq:fsub_upp}, and setting $\softer = 2\optgap^{-1}\log\fcount$ as in Algorithm \ref{alg:smooth} for a given optimality gap guarantee $\optgap$, we get
	\begin{equation} \label{eq:f_bound}
		\func(\argu_t) \leq 5\radius^2 + \optgap/2,
		\quad \text{ for all }t\geq 1.
	\end{equation}
	
	Regarding the gradients for minimal bounding sphere problem, using the expectation form of the gradient in Lemma \ref{lemma:sub_grad} and incorporating the function definition in \eqref{eq:func_mbs}, we have:
	\begin{equation} \label{eq:alt_grad}
	\grad \fsub_\softer(\argu) = 2\left(\argu - \E{\prob_\softer(\argu)}{\point_i}\right)
	\end{equation}
	
	Combining \eqref{eq:f_bound} and \eqref{eq:alt_grad}, we have a bound on the gradient norms of the smoothing function $\fsub_\softer(\cdot)$ at points $\{\argu_{t}\}_{t=1}^\infty$ as
	\begin{equation} \label{eq:grad_x_bnd}
	\|\grad \fsub_\softer(\argu_t)\| \leq 2\sqrt{5\radius^2 + \optgap/2},
	\end{equation}
	since we can claim $\|\argu_{t} - \E{\prob_\softer(\argu_t)}{\point_i}\| \leq \sqrt{\func(\argu_{t})}$, which results from the distance between $\argu_{t}$ and some weighted average of points $\point_i$, specifically $\E{\prob_\softer(\argu_t)}{\point_i}$, being at most the distance between $\argu_{t}$ and the point $\point_i$ farthest to it, i.e. $\sqrt{\func(\argu_{t})}$.
	
	Let us next investigate the gradients at $\{\altarg_t\}_{t=1}^\infty$, which are calculated on Line \ref{line:x} of Algorithm \ref{alg:smooth}.
	
	For $t=1$, its norm is upper-bounded by $2\radius$ since the diameter of minimal bounding sphere is $2\radius$ which includes the initialization $\altarg_1$. For $t>1$, combining \eqref{eq:alt_grad} and Line \ref{line:y} from Algorithm \ref{alg:smooth}, we have
	\begin{equation*}
		\frac{1}{2}\cdot \grad \fsub_\softer(\altarg_t) 
		= \argu_{t} + \left(
		1 - 2\left(\sqrt{\cond_\softer} + 1\right)^{-1}
		\right) \left(\argu_{t} - \argu_{t-1}\right) 
		- \E{\prob_\softer(\altarg_t)}{\point_i}.
	\end{equation*} 
	Using the triangle inequality and by upper bounding the negative terms with 0,
	\begin{equation*}
		\|\grad \fsub_\softer(\altarg_t)\| \leq 4\|\argu_{t} - \E{\prob_\softer(\altarg_t)}{\point_i}\|
		+ 2\|\argu_{t-1} - \E{\prob_\softer(\altarg_t)}{\point_i}\|.
	\end{equation*} 
	Finally, using \eqref{eq:f_bound}, we have
	\begin{equation} \label{eq:grad_y_bnd}
		\|\grad \fsub_\softer(\altarg_t)\| \leq 6\sqrt{5\radius^2 + \optgap/2},
	\end{equation}
	as we can claim $\|\argu_{t} - \E{\prob_\softer(\altarg_t)}{\point_i}\| \leq \sqrt{\func(\argu_{t})}$ like before.
	
	With \eqref{eq:grad_x_bnd} and \eqref{eq:grad_y_bnd}, we have bounded the gradient norms at all iterations $\{\argu_{t}\}_{t=1}^\infty$ and $\{\altarg_{t}\}_{t=1}^\infty$. We take an arbitrary $\argu$, member to the convex-hull of $\{\argu_{t}\}_{t=1}^\infty$, $\{\altarg_{t}\}_{t=1}^\infty$ and the optimal point $\argu^*$. As discussed in Section \ref{sec:acce}, it is sufficient to generate a gradient norm upper bound for this arbitrary point $\argu$ to obtain $\gradupp_\softer$. Since $\argu$ is a convex combination of $\{\argu_{t}\}_{t=1}^\infty$, $\{\altarg_{t}\}_{t=1}^\infty$ and $\argu^*$, we decompose it into individual parts and insert that version of $\argu$ into \eqref{eq:alt_grad}. Using triangle inequality and the claim $\|\argu - \E{\prob_\softer(\altarg)}{\point_i}\| \leq \sqrt{\func(\argu)}$ for any pair of $(\argu,\altarg)$, the common gradient norm bound turns out to be the maximum of bounds \eqref{eq:grad_x_bnd} and \eqref{eq:grad_y_bnd}, since the gradient at optimal point is 0.
	Consequently, we can set $\gradupp_\softer = 6\sqrt{5\func(\argu_1)+\optgap/2}$.
\end{proof}

\subsection{Convergence result}
Before examining the convergence result, we note that, for minimal bounding sphere problem, the $(1+\arbsmall)$-approximation translates into converging to a bounding sphere with radius $(1+\arbsmall)\radius$. Consequently, we have that, for some $t$:
\begin{equation*}
	\func(\argu_t) - \func(\argu^*) \leq
		(1+\arbsmall)^2\radius^2-\radius^2 \leq \left(2\arbsmall + \arbsmall^2\right) R^2,
\end{equation*}
meaning the requested optimality gap $\optgap = \left(2\arbsmall + \arbsmall^2\right) R^2$, i.e $\optgap \in O(\arbsmall \radius^2)$ for $0 < \arbsmall \leq 1$.

\begin{theorem}
	For the minimal bounding sphere problem, we can generate an approximate solution by achieving a bounding sphere with radius $(1+\arbsmall)\radius$ for an arbitrarily small positive $\arbsmall$ using Algorithm \ref{alg:smooth}. After setting $\heslow_{\softer,i} = \hesupp_{\softer,i} = 2$ for all $1\leq i\leq \fcount$, and $\gradupp_\softer = 6\sqrt{5\func(\argu_1)+\optgap/2}$, the overall computational complexity $T$ and the total number of iterations by the algorithm are
	\begin{equation*}
		 \begin{gathered}
			T \in \tilde{O}\left(\fcount\dimcnt \sqrt{\frac{1}{\arbsmall}}\right),
			\\\text{ and }
			t = 1 + 
				\log\left(
					1 + \frac{4}{\arbsmall}
				\right)
				\sqrt{
					1 + 18\left(
						1 + \frac{20}{\arbsmall}
					\right)\log\fcount
				}
		\end{gathered}
	\end{equation*}
	where $\tilde{O}(\cdot)$ is the soft-O notation which ignores the logarithmic additives and multipliers.
\end{theorem}
\begin{proof}
	We plug in the  $\heslow_{\softer,i} = \hesupp_{\softer,i} = 2$ for all $1\leq i\leq \fcount$ and $\gradupp_\softer = 6\sqrt{5\func(\argu_1)+\optgap/2}$ into the result of Theorem \ref{thm:itecnt} regarding the number of iterations required. We can upper bound the right side of the equality for $t$ before using here, since it can only provide further guarantees as shown in Lemma \ref{lemma:opt_gap}. We also plug in the initial distance $\dist_\softer \leq \radius^2$ by definition of $\radius$, the radius of the minimal bounding sphere, and the selection of $\argu_1$ from the convex-hull of $\{\point_i\}_{i=1}^\fcount$. We note that $\func(\argu_1) \leq 4\radius^2$ from Lemma \ref{lemma:radius_bounds} and bound $\func(\argu_1)$ with $4\radius^2$. Instead of upper bounding $\fsub_\softer(\argu_1) - \fsub_\softer(\argu^*)$ by $\gradupp_\softer \dist_\softer$, as previously done in Theorem \ref{thm:itecnt}, we can use the upper bound of $\left(3\radius^2 + \optgap/2\right)$ resulting from Corollary \ref{cor:gap} and setting $\softer = 2 \optgap^{-1} \log\fcount$ since we have $\func(\argu_1) - \func(\argu^*) \leq 3\radius^2$ due to $\func(\argu^*)=\radius^2$. Lastly, we upper bound the reciprocal of optimality gap, i.e. $1/\optgap$ with $1/(2\arbsmall\radius^2)$ since $\arbsmall > 0$.
\end{proof}

\newpage
\bibliographystyle{IEEEtran}
\bibliography{IEEEabrv,single}

\end{document}